\documentclass{article}
\usepackage{amssymb,latexsym}
\newcommand{\Rn}{\mathbb{R}}
\newcommand{\Nn}{\mathbb{N}}
\newcommand{\Jac}{\mathop{\mathrm{Jac}}}
\newtheorem{Lemma}{Lemma}
\newtheorem{Theorem}{Theorem}

\newtheorem{Corollary}{Corollary}
\newtheorem{Example}{Example}
\newenvironment{proof}[1][Proof]{\textbf{#1.} }{\
\rule{0.5em}{0.5em}}

\begin{document}
\author{Janusz Gwo\'zdziewicz}
\title{Real Jacobian mates}
\maketitle


\begin{abstract}
Let $p$ be a real polynomial in two variables. 
We say that a polynomial~$q$ is a real Jacobian mate 
of $p$ if the Jacobian determinant of the mapping 
$(p,q):\mathbb{R}^2\to\mathbb{R}^2$
is everywhere positive. We present a class of polynomials that do 
not have real Jacobian mates. 
\end{abstract}

\section{Introduction}
This note is inspired by \cite{BF} where  Braun and dos Santos Filho proved
that every polynomial mapping $(p,q):\Rn^2\to\Rn^2$ with everywhere positive 
Jacobian determinant and such that $\deg p\leq 3$ is a global diffeomorphism.  

A pair of polynomials $p,q\in\Rn[x,y]$ such that the mapping $(p,q):\Rn^2\to\Rn^2$ 
has everywhere positive Jacobian will be called \emph{real Jacobian mates}. 
The key role in \cite{BF} plays the result that $p=x(1+xy)$ does not have a 
real Jacobian mate. 

This result is a special case of Theorem~\ref{glacial}.  
In Theorem~\ref{branches} a wide class of polynomials that 
do not have real Jacobian mates is characterized.  
In particular every polynomial such that its Newton polygon 
has an edge described in Corollary~\ref{edge} belong to this class. 
This gives a new proof of  \cite[Theorem~5.5]{BO} that polynomials 
of degree~4 with at least one disconnected level set do not have real Jacobian mates. 

\section{Glacial tongues}
\begin{Theorem}\label{glacial}
Let $p$ be a real polynomial in two variables
and let $B\subset A$, be the subsets of the real plane such that: 
\begin{itemize}
\item[\emph{(i)}] the set $B$ is bounded,
\item[\emph{(ii)}] for every $t\in\Rn$ the set $p^{-1}(t)\cap A$ is either empty,
or is contained in $B$, 
or is homeomorphic to a segment and its endpoints belong to $B$, 
\item[\emph{(iii)}] the border of $A$ contains a half-line.
\end{itemize}
Then for every $q\in\Rn[x,y]$ there exists $v\in\Rn^2$, such that  $\Jac(p,q)(v)=0$.
\end{Theorem}

\begin{proof} 
Suppose that there exists 
$q\in\Rn[x,y]$ such that $\Jac(p,q)$ vanishes nowhere.
Under this assumption the mapping $\Phi=(p,q):\Rn^2\to\Rn^2$ is a local diffeomorphism.

Take any $t\in\Rn$ such that the set $A_t=p^{-1}(t)\cap A$ is nonempty.
If $A_t\subset B$ then $\Phi(A_t)\subset\Phi(B)$.
If $A_t$ is homeomorphic to a segment with endpoints in $B$
then the restriction of $\Phi$ to $A_t$ is a locally injective continuous mapping 
from the source $A_t$ which is homeomorphic to a segment to a vertical line $\{t\}\times\Rn$
which is homeomorphic to $\Rn$.  
By the extreme value theorem and the mean value theorem 
such a mapping is either increasing or decreasing.%
\footnote{Suppose that a continuous 
and locally injective function $f:[a,b]\to\Rn$ in neither increasing nor decreasing. 
Then there exist $x_1$, $x_2$ $x_3$, $a\leq x_1<x_2<x_3\leq b$ such that 
$f(x_1)\leq f(x_2) \geq f(x_3)$ or $f(x_1)\geq f(x_2) \leq f(x_3)$. 
By the extreme value theorem $f$ restricted to $[x_1,x_3]$ has a maximal or a minimal value 
at some point $c$ inside the interval $[x_1,x_3]$. Shrinking $[x_1,x_3]$, if necessary, 
we may assume that $f$ restricted to $[x_1,x_3]$ is injective. 
By the mean value theorem $f(x_3)\in f([x_1,c])$ or $f(x_1)\in f([c,x_3])$ which 
gives a contradiction.} 
Hence, the image $\Phi(A_t)$ is a vertical segment with endpoints that belong to $\Phi(B)$. 

Since $A$ is the union of the sets $A_t$ and $\Phi(B)$ is bounded, so it is $\Phi(A)$. 

\smallskip
Let $L$ be a half-line contained in the border of $A$. 
Because the mapping $\Phi$ is bounded on $A$ it is also 
bounded on $L$. Consequently the polynomials
$p$ and $q$ restricted to $L$ are constant
(because they behave on $L$ like polynomials in one variable). 
We arrived to a contradiction with the condition that $\Phi$ is locally injective.
\end{proof}

\medskip
Every set $A$ satisfying assumptions of Theorem~\ref{glacial}
will be called \textit{a glacial tongue with a straight border}. 

\begin{Example} 
Let $p=x(1+xy)$. In~\cite{BF} it is checked (Lemma~4.1 and Remark~1) that 
$A=\{(x,y)\in\Rn^2: 0<x<1, -\frac{1}{x}<y\leq -1\}$ is a glacial tongue with a 
straight border for the polynomial $p$.  Hence, $p$ does not have a Jacobian mate.  
\end{Example}


\section{Newton polygon and branches at infinity}

Let $p=\sum a_{i,j}x^iy^j$ be a nonzero polynomial. 
By definition the Newton polygon $\Delta(p)$ is the convex hull of the set $\{(i,j)\in\Nn^2:a_{i,j}\neq0\}$. 
An edge $S$ of the Newton polygon $\Delta(p)$ will be called an outer edge if it has a normal vector $\vec v=(v_1,v_2)$ directed outwards $\Delta(p)$ such that $v_1>0$ or $v_2>0$.  
If $v_1>0$ then $S$ will be called a right outer edge. With every right outer edge $S$
we associate a rational number $\theta(S)=v_2/v_1$ and call this number the slope of $S$. 

\begin{Example} 
The Newton polygon of $p=x+x^2+x^3y+y^2+x^3y^2+xy^4$ has $4$ outer edges. 
Three of them are right outer edges with slopes $-1$, $0$, and $2$.

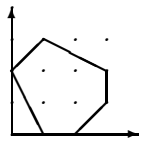
\setlength{\unitlength}{12pt}
\begin{picture}(6,5)(0,0)
\thinlines
\put(0,0){\vector(1,0){4}}    \put(0,0){\vector(0,1){4}}
\put(0,0){\makebox(0,0){.}} \put(1,0){\makebox(0,0){.}}
\put(2,0){\makebox(0,0){.}} \put(3,0){\makebox(0,0){.}}
\put(0,1){\makebox(0,0){.}} \put(1,1){\makebox(0,0){.}} 
\put(2,1){\makebox(0,0){.}} \put(3,1){\makebox(0,0){.}} 
\put(0,2){\makebox(0,0){.}} \put(1,2){\makebox(0,0){.}} 
\put(2,2){\makebox(0,0){.}} \put(3,2){\makebox(0,0){.}}
\put(0,3){\makebox(0,0){.}} \put(1,3){\makebox(0,0){.}}
\put(2,3){\makebox(0,0){.}} \put(0,4){\makebox(0,0){.}} 
\put(3,3){\makebox(0,0){.}}
\thicklines
\put(0,2){\line(1,1){1}}
\put(1,3){\line(2,-1){2}}
\put(3,1){\line(0,1){1}}
\put(2,0){\line(1,1){1}}
\put(0,2){\line(1,-2){1}}
\put(1,0){\line(1,0){1}}
\end{picture}
\end{Example}

\medskip
The objective of this section is to describe branches at infinity of a curve $p(x,y)=0$ 
and associate with each branch a certain outer edge of the Newton polygon of $p$. 

Let $V=\{(x,y)\in\Rn^2:p(x,y)=0\}$. Assume that the curve $V$ is unbounded and consider 
a one-point algebraic compactification $\widehat \Rn^2=\Rn^2\cup\{\infty\}$ of the real plane
(see \cite[Definition~3.6.12]{BR}).
Then $\infty$ belongs to the Zariski closure of $V$ in $\widehat \Rn^2$.  By \cite[Lemma~3.3]{Mi}
in a suitably chosen neighborhood of $\infty$ the curve $V\cup\{\infty\}$ is the union of finitely many branches which intersect only at $\infty$. Each branch is homeomorphic to an open interval under an 
analytic homeomorphism $p:(-\epsilon,\epsilon)\to V\cup\{\infty\}$, $p(0)=\infty$.  

It follows from the above that after passing to coordinates $x$ and $y$ in $\Rn^2$ and substituting 
$s=t^{-1}$ in $p$ we get the following characterization of branches at infinity. 

\begin{Lemma}\label{structure}
Assume that  $V=\{(x,y)\in\Rn^2:p(x,y)=0\}$ is an unbounded polynomial curve. 
Then in a suitably chosen neighborhood of infinity in $\Rn^2$
the curve $V$ is the union of finitely many 
pairwise disjoint ``branches at infinity''. Each branch at infinity 
is homeomorphic to a union of 
two open intervals $(-\infty,-R)\cup(R,+\infty)$ under a homeomorphism 
$(x,y)=(\tilde x(t),\tilde y(t))$ which is given by Laurent power series 
\begin{eqnarray}
\tilde x(t)&=&a_kt^k+a_{k-1}t^{k-1}+a_{k-2}t^{k-2}+\cdots  \label{eq:x} \\
\tilde y(t)&=&b_lt^l+b_{l-1}t^{l-1}+b_{l-2}t^{l-2}+\cdots  \label{eq:y}
\end{eqnarray}
convergent for $|x|>R$.
\end{Lemma}

\begin{Lemma}\label{edges}
Keep the assumptions an notations of Lemma~\ref{structure}. 
If $a_k\neq0$, $b_l\neq0$ then $(k,l)$ is a normal vector 
to some outer edge of the Newton polygon of~$p$. 
\end{Lemma}

\begin{proof}
Let $d=\max \{ki+lj:(i,j)\in\Delta(p)\}$. The polynomial $p$ can be written as a sum
$p=\sum_{ki+lj\leq d}c_{i,j}x^iy^j$.
Substituting $(x,y)=(\tilde x(t),\tilde y(t))$ to $p$ 
and collecting together the terms of the highest degree we get 
$$ 0=p(\tilde x(t),\tilde y(t))=
    \Bigl( \sum_{ki+lj=d}c_{i,j}a_k^ib_l^j\Bigr)t^d+\mbox{ terms of lower degrees}.
$$
The necessary condition for this identity is a cancellation of terms in the sum in parenthesis. 
If there are at least two distinct coefficients $c_{i_1,j_1}\neq0$, $c_{i_2,j_2}\neq0$, 
satisfying $ki_1+lj_1=ki_2+lj_2=d$ then the straight  line $\{(i,j)\in\Rn^2:ki+lj=d\}$ touches $\Delta(p)$ at a least 
two points, hence along the edge. 

Since $(x,y)=(\tilde x(t),\tilde y(t))$ is a Laurent parametrization of a branch at infinity, we have
$\|(\tilde x(t),\tilde y(t))\|\to\infty$ as $t\to+\infty$ which proves that at least one of exponents 
$k$, $l$ is positive and shows that $\Delta(p)\cap\{(i,j)\in\Rn^2:ki+lj=d\}$ is an outer edge.
\end{proof}

\medskip
Using Lemma~\ref{edges} 
we may associate to every branch at infinity of the curve $p=0$ the
unique outer edge of the Newton polygon of $p$. 
In the next lemma we will show that the slope of the associated edge characterizes 
the asymptotic of the branch at infinity. 

\medskip

For two real valued functions $g$, $h$ defined in some interval $(R,\infty)$ we will write 
$g(x)\sim h(x)$ if there exist constants $c>0$, $C>0$, and $r>0$ such that 
$c |h(x)|\leq |g(x)| \leq C|h(x)|$ for all $x>r$.

\begin{Lemma}\label{structure1}
Let $p(x,y)$ be a nonzero real polynomial such  that for every~$x_0$ the set 
$X=\{\,(x,y)\in\Rn^2: x>x_0,\,y>0,\,p(x,y)=0\,\}$ is nonempty. 
Then for sufficiently large $x_0$ there exists a finite collection of continuous 
semialgebraic functions 
$f_k:(x_0,+\infty)\to\Rn$, $k=1,\dots, s$ such that 
\begin{itemize}
\item[\emph{(i)}] $0<f_1(x)<\dots<f_s(x)$ for $x>x_0$, 
\item[\emph{(ii)}] $X$ is the union of graphs $\{\,(x,y)\in\Rn^2: y=f_k(x),\,x>x_0\,\}$, 
$k=1,\dots, s$,
\item[\emph{(iii)}] for every $f_k$ there exists a right outer edge $S_k$ of the Newton 
                            polygon of $p(x,y)$ such that  $f_k(x)\sim x^{\theta(S_k)}$.
\end{itemize}
\end{Lemma}

\begin{proof} 
Part~(i) and~(ii) follow from the Cylindrical Decomposition Theorem for semialgebraic sets 
(see for example \cite[Theorem~2.2.1]{BR}).  

To prove~(iii) observe that the graph of $f_k$ is unbounded and homeomorphic to an open interval. 
Thus, we may assume, increasing $x_0$ if necessary, that this graph is a half-branch at infinity.
By Lemma~\ref{structure} there exists a homeomorphism of an open interval $(R,+\infty)$ 
to the graph given by Laurent power series~(\ref{eq:x}),~(\ref{eq:y})
with $a_k\neq0$, $b_l\neq0$.  Since $\tilde x(t)\to+\infty$ for $t\to+\infty$,
the leading term of  $\tilde x(t)$ has a positive exponent $k$. 
By estimations  $\tilde x(t) \sim t^k$,  $\tilde y(t) \sim t^l$
and identity $f_k(\tilde x(t))=\tilde y(t)$ we get 
$f_k(x)\sim x^{l/k}$.  
Finally, by Lemma~\ref{edges} there exists a right outer 
edge $S$ of the Newton polygon of $p$ such that $l/k=\theta(S)$.
\end{proof}

\section{Main result}

\begin{Theorem}\label{branches}
Assume that the Newton polygon of $p\in\Rn[x,y]$ has a right outer edge $S$ 
with endpoint $(0,1)$ and positive inclination and that the curve $p=0$ 
has a real branch at infinity associated with the edge $S$. 
Then $p$ has a glacial tongue with a straight border.
\end{Theorem}


\begin{proof}
Without loss of generality we may assume, changing signs of variables if necessary,  
that one of half-branches associated with the edge $S$ lies in the positive quadrant $x>0$, $y>0$.  

Then, under notation of Lemma~\ref{structure1} this half-branch at infinity is a graph $y=f(x)$
where $f$ is one of the functions $f_k$, $k=1,\dots,s$. 
Comparing the asymptotic of these functions 
we see that $\theta(S_1)\leq\theta(S_2)\leq\dots\leq \theta(S_s)$.
Since $S$ has the smallest slope among all right outer edges of the Newton polygon $\Delta(p)$, 
we have $S=S_1$ and we may assume that $f(x)=f_1(x)$. 

Let 
$$V=\{(x,y)\in\Rn^2:x>x_0,\, 0<y<f(x)\} .$$
The polynomial $p$ vanishes nowhere on $V$, 
hence without loss of generality we may assume that $p$ is positive on this set.  

\medskip
\noindent
\textbf{Claim~1.} For every $t\neq0$ the set $p^{-1}(t)\cap V$ is bounded. 

\smallskip
\textbf{Proof of 1.}
If not, then by the Curve Selection Lemma there exists a half-branch at infinity of a curve $p(x,y)=t$ 
contained in $V$. Let $y=g(x)$ be the graph of this half-branch at infinity. By Lemma~3  
$g(x)\sim x^{\theta(S_1)}$, where $S_1$ is one of the right outer edges of the Newton polygon  $\Delta(p-t)$.
By inequalities $0<g(x)<f(x)$ we get $\theta(S_1)\leq \theta(S)$.  
This is impossible because all right outer edges of $\Delta(p-t)$ have slopes bigger 
than the slope of $S$.

\medskip
\noindent
\textbf{Claim~2.} For $x_0$ sufficiently large, $V$ does not contain any critical point of~$p$.

\smallskip
\textbf{Proof of 2.} 
If the intersection of $V$ with the set of critical points 
is bounded then it is enough to enlarge $x_0$. If this intersection is unbounded then 
by the Curve Selection Lemma it contains  an unbounded semi-algebraic arc $\Gamma\subset V$. 
It follows that $p$ restricted to $\Gamma$ is constant and nonzero -- contrary to Claim~1.

\medskip
Further, we will assume that $V$ satisfies assumptions of Claim~2.
Then every level set $p^{-1}(t)$ intersected with $V$ 
is a one-dimensional smooth semialgebraic manifold.  
By Poincare-Bendixon Theorem 
$V_t=p^{-1}(t)\cap V$ has a finite number of connected 
components, each homeomorphic to a circle or to an open interval. 

\medskip
\noindent
\textbf{Claim~3.}
 There is no connected component of $V_t$ homeomorphic to a circle. 

\smallskip
\textbf{Proof of 3.} 
Suppose there is. Then by Jordan's Theorem it cuts the set $V$ to two open regions. 
One of these regions is bounded. Since the function $p$ is constant on the boundary of this region, 
it attains an extreme value at 
some point inside.  This is impossible because $p$ has no critical points in the set $V$. 

\medskip
Let $h(y)=p(x_0,y)$ be the restriction of $p$ to the vertical line $\{x_0\}\times \Rn$.
A function $h$ vanishes at the endpoints of the interval $[0,f(x_0)]$ and is positive
inside.  It is easy to find $t_0>0$ and two points $a<b$ inside the interval $[0,f(x_0)]$ such that:\\
$h'(y)\neq0$ for $y\in(0,a]\cup [b,f(x_0))$,\\
$h$ increases from $0$ to $t_0$ in the interval $[0,a]$, \\
$h(y)>t_0$ for  $a<y<b$, \\
$h$ decreases from $t_0$ to $0$ in the interval $[b,f(x_0)]$.

\medskip
\noindent
\textbf{Claim~4.}
For every $t$ such that $0<t\leq t_0$ the set $V_t=p^{-1}(t)\cap V$ 
is connected and homeomorphic to an open interval. 
The topological closure of $V_t$ intersects  the vertical segment 
$\{x_0\}\times (0,f(x_0))$ at two points. 

%
\smallskip
\textbf{Proof of 4.} 
By the discussion proceeding Claim~4 the polynomial~$p$ attains value~$t$ 
precisely at two points of the boundary of $V$. 
These are $(x_0,y_1)$ and $(x_0,y_2)$, where $0<y_1\leq a$ and $b\leq y_2<f(x_0)$.
Moreover $\partial p/\partial y$ does not vanish at these points. 

By Claim~2 and Claim~3 the set $V_t$ is a one-dimensional smooth manifold having a 
finite number of connected component; each component 
is semialgebraic and homeomorphic to an open interval. Thus, the closure of $V_t$ is a graph 
with vertexes  $(x_0,y_1)$, $(x_0,y_2)$ and edges which are connected components of $V_t$. 

By the Implicit Function Theorem the closure of $V_t$ has in a small neighborhood of
$(x_0,y_i)$, where $i=1,2$  a topological type of an interval $[0,1)$ which shows that 
there is exactly one edge which connects $(x_0,y_1)$ and $(x_0,y_2)$.

\medskip
By Claim~4 the closure of $V_{t_0}$ is a line with two endpoints: $(x_0,a)$ and $(x_0,b)$.
Joining them by a vertical segment we get a non-smooth oval. By Jordan's Theorem 
this oval cuts the plane into two open regions.  Let $B_0$ be the bounded region,
let $B=B_0\cup\{x_0\}\times (0,f(x_0))$
and let $A=V\cup\{x_0\}\times (0,f(x_0))$.


\smallskip
If $t\leq 0$ then $A_t$ is empty. 
If $0<t\leq t_0$ then $A_t$ is homeomorphic to a line with endpoints at $\{x_0\}\times (0,f(x_0))$.
If $t>t_0$ then either $A_t$ is empty or the  closure of every connected component of $A_t$ 
intersects the border of $A$ along $x_0\times(a,b)$. In this case $A_t\subset B$. 
\end{proof}


\begin{Corollary}\label{edge}
Assume that the Newton polygon of a polynomial $p\in\Rn[x,y]$ has a right outer edge that: 
begins at $(0,1)$, has a positive inclination, and its only lattice points are the endpoints. 
Then $p$~does not have a real Jacobian mate. 
\end{Corollary}

\begin{proof} 
It is enough to prove that there exists a branch at infinity of the curve $p=0$ associated 
with the edge $S$ satisfying assumptions of Corollary~\ref{edge}. 
Let $(0,1)$, $(a,b)$ be the endpoints of  $S$. 
Then the polynomial $p$ has two terms $Ax^ay^b$ and $By$
corresponding to the lattice points of $S$. Multiplying $x$, $y$ and $p$ by nonzero constants, 
we may reduce our considerations to $A=1$ and $B=-1$.   
Substituting $(x(t),y(t))=(ct^{b-1},t^{-a})$ 
we get  $p(x(t),y(t))=(c^a-1)t^{-a} +  \mbox{ terms of lower degrees}$. Hence, the sign 
of the polynomial $p$ on the curve $(x(t),y(t))$ for large $t$ depends on the sign of $c^a-1$. 
The curve $(x(t),y(t))$ for $t>0$ is a graph of a function. 
By the appropriate choice of $c$ we can find two functions $f_1$ $f_2$ such that 
$0<f_1<f_2$, $f_1(x) \sim f_2(x) \sim x^{\theta(S)}$, $p$ has negative values 
on the graph of $f_1$ and has positive values on the graph of $f_2$. By Lemma~3 this can happen if and only if there is a half-branch at infinity of the curve $p=0$ which is a graph of a function $g$ with 
$g(x)\sim x^{\theta(S)}$.
\end{proof}

\medskip
\noindent
\textbf{Remark.} 
Using toric modifications of the real plane one can present a shorter proof of Corollary~1.

\begin{Example}\label{list}
Every polynomial from the list: 
$p_1=y+xy^2+y^4$,
$p_2=y+ay^2+xy^3$,
$p_3=y+x^2y^2$,
$p_4=y+ay^2+y^3+x^2y^2$ satisfies assumptions of Corollary~\ref{edge}. 

The Newton polygons of these polynomials are drawn below.
\end{Example}

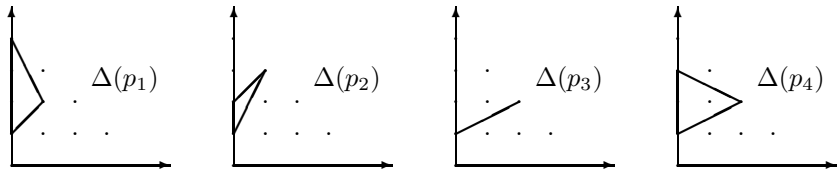
\setlength{\unitlength}{12pt}
\begin{picture}(24,5)(0,0)
\newsavebox{\sn}
\thinlines
\savebox{\sn}(6,5){\put(0,0){\vector(1,0){5}}  \put(0,0){\vector(0,1){5}}
\put(0,0){\makebox(0,0){.}} 
\put(1,0){\makebox(0,0){.}}
\put(2,0){\makebox(0,0){.}} 
\put(3,0){\makebox(0,0){.}}
\put(0,1){\makebox(0,0){.}}
\put(1,1){\makebox(0,0){.}} 
\put(2,1){\makebox(0,0){.}}
\put(3,1){\makebox(0,0){.}} 
\put(0,2){\makebox(0,0){.}}
\put(1,2){\makebox(0,0){.}} 
\put(2,2){\makebox(0,0){.}}
\put(0,3){\makebox(0,0){.}} 
\put(1,3){\makebox(0,0){.}}
\put(0,4){\makebox(0,0){.}} 
}

\put(0,0){\usebox{\sn}}
\put(0,0)
{\begin{picture}(5,4)(0,0)
\put(2.5,2.5){$\Delta(p_1)$}
\thicklines
\put(0,1){\line(1,1){1}}
\put(0,4){\line(1,-2){1}}
\put(0,1){\line(0,1){3}}
\end{picture}}

\put(7,0){\usebox{\sn}}
\put(7,0)
{\begin{picture}(5,4)(0,0)
\put(2.5,2.5){$\Delta(p_2)$}
\thicklines
\put(0,1){\line(1,2){1}}
\put(0,2){\line(1,1){1}}
\put(0,1){\line(0,1){1}}
\end{picture}}

\put(14,0){\usebox{\sn}}
\put(14,0)
{\begin{picture}(5,4)(0,0)
\thicklines
\put(2.5,2.5){$\Delta(p_3)$}
\put(0,1){\line(2,1){2}}
\end{picture}}

\put(21,0){\usebox{\sn}}
\put(21,0)
{\begin{picture}(5,4)(0,0)
\put(2.5,2.5){$\Delta(p_4)$}
\thicklines
\put(0,1){\line(2,1){2}}
\put(0,3){\line(2,-1){2}}
\put(0,1){\line(0,1){2}}
\end{picture}}
\end{picture}

\bigskip
The polynomials in the above example are taken from~\cite{BO}. 
Theorem~1.3 in the cited paper states that these polynomials are canonical forms, 
up to affine substitution of polynomials of degree~4 
without critical points and with at least one disconnected level set. 
Theorem~5.5 says that none of these polynomials has a real Jacobian mate. 
The method of its proof uses an integration  based on Green's formula 
and requires an analysis of each case separately.


\begin{thebibliography}{99}


\bibitem{BR} B. Benedettini, J.~J.~Risler, 
{\em Real algebraic and semi-algebraic sets}, Hermann, Paris, 1990.

\bibitem{BF} F. Braun, J. R. dos Santos Filho, 
{\em The real jacobian conjecture on $\Rn^2$ is true when one of the components has degree 3}, 
Discrete Contin. Dyn. Syst. 26, 75--87 (2010)

\bibitem{BO} F.~Braun,  B.~Or\'efice-Okamoto,  
{\em On polynomial submersions of degree $4$ and the real jacobian conjecture in $\Rn^2$}, 
arXiv:1406.7683

\bibitem{Mi} J.~W.~Milnor, 
{\em Singular Points of Complex Hypersurfaces}, 
Princeton Univ. Press, Princeton, 1968.

%

\end{thebibliography}
\end{document}